\documentclass{article}
\usepackage{fullpage}
\usepackage[utf8]{inputenc}

\usepackage[margin=1truein]{geometry}

\usepackage[T1]{fontenc}

\usepackage{amsmath}
\usepackage{amssymb}
\usepackage{amsthm}
\usepackage{cite}
\usepackage{comment}
\usepackage{enumitem}
\usepackage{physics}
\usepackage{mleftright}\mleftright
\usepackage{mathtools}
\DeclarePairedDelimiter{\prn}{(}{)}
\DeclarePairedDelimiter{\set}{\{}{\}}
\DeclarePairedDelimiterX{\Set}[2]{\{}{\}}{\,{#1}\;\delimsize|\;{#2}\,}
\newcommand{\symdif}{\mathbin{\triangle}}
\usepackage{nicematrix}

\newcommand{\Mod}[1]{\ (\mathrm{mod}\ #1)}

\usepackage[
    bookmarksnumbered,
    colorlinks,
    hypertexnames=false,
    pdfdisplaydoctitle,
    pdfusetitle,
    unicode
]{hyperref}

\usepackage{upref}
\usepackage[capitalize,noabbrev]{cleveref}

\crefname{conjecture}{Conjecture}{Conjectures}
\crefname{claim}{Claim}{Claims}

\newtheorem{theorem}{Theorem}[section]
\newtheorem{lemma}[theorem]{Lemma}
\newtheorem{proposition}[theorem]{Proposition}
\newtheorem{claim}[theorem]{Claim}
\newtheorem{conjecture}[theorem]{Conjecture}
\newtheorem{question}[theorem]{Question}
\newtheorem{corollary}[theorem]{Corollary}
\theoremstyle{definition}
\newtheorem{example}[theorem]{Example}
\newtheorem{remark}[theorem]{Remark}

\usepackage{thm-restate}

\expandafter\def\csname ver@etex.sty\endcsname{3000/12/31}
\usepackage{autonum}

\usepackage{listings}
\usepackage{color}
\newcommand{\GF}{\mathrm{GF}}
\newcommand{\cups}{\cup\dots\cup}
\newcommand{\cH}{\mathcal{H}}
\newcommand{\N}{\mathbb{N}}
\newcommand{\Z}{\mathbb{Z}}
\newcommand{\R}{\mathbb{R}}

\def\final{1}  % set this to 1 to get a comment-free version
\def\iflong{\iffalse}
\ifnum\final=0  %namely if we allow comments in the output
\newcommand{\dnote}[1]{{\color{orange}[{\tiny \textbf{Dani:} \bf #1}]\marginpar{\color{orange}*}}}
\newcommand{\ynote}[1]{{\color{red}[{\tiny \textbf{Yutaro:} \bf #1}]\marginpar{\color{red}*}}}
\newcommand{\rnote}[1]{{\color{teal}[{\tiny \textbf{Ryuhei:} \bf #1}]\marginpar{\color{teal}*}}}
\newcommand{\onote}[1]{{\color{purple}[{\tiny \textbf{Taihei:} \bf #1}]\marginpar{\color{purple}*}}}
\newcommand{\snote}[1]{{\color{cyan}[{\tiny \textbf{Tamás:} \bf #1}]\marginpar{\color{cyan}*}}}

\else % in this case [final=1] we don't want any comments to show
\newcommand{\dnote}[1]{}
\newcommand{\ynote}[1]{}
\newcommand{\rnote}[1]{}
\newcommand{\onote}[1]{}
\newcommand{\snote}[1]{}
\fi

\newif\ifshortintro
\shortintrofalse

\begin{document}

\title{\texorpdfstring{%
  Towards the Proximity Conjecture on Group-Labeled Matroids%
}{%
  Towards the Proximity Conjecture on Group-Labeled Matroids%
}}
\author{
\hspace{5mm}
Dániel Garamvölgyi\thanks{MTA-ELTE Matroid Optimization Research Group, Department of Operations Research,
Eötvös Loránd University, and HUN-REN Alfréd Rényi Institute of Mathematics, Budapest, Hungary. E-mail: \texttt{daniel.garamvolgyi@ttk.elte.hu}.}
\hspace{5mm}
\and
Ryuhei Mizutani\thanks{Department of Mathematical Informatics, Graduate School of Information Science and Technology, The University of Tokyo, Tokyo, Japan. E-mail: \texttt{ryuhei\_mizutani@mist.i.u-tokyo.ac.jp}.}
\hspace{5mm}
\and
Taihei Oki\thanks{Institute for Chemical Reaction Design and Discovery (ICReDD), Hokkaido University, Sapporo, Hokkaido, Japan. E-mail: \texttt{oki@icredd.hokudai.ac.jp}.}
\hspace{5mm}
\and
Tamás Schwarcz\thanks{MTA-ELTE Matroid Optimization Research Group and HUN-REN–ELTE Egerváry Research Group, Department of
Operations Research, Eötvös Loránd University, Budapest, Hungary. E-mail: \texttt{tamas.schwarcz@ttk.elte.hu}.}
\and 
Yutaro Yamaguchi\thanks{Department of Information and Physical Sciences, Graduate School of Information Science and Technology, Osaka University, Osaka, Japan. E-mail: \texttt{yutaro.yamaguchi@ist.osaka-u.ac.jp}.}
}
\date{}

\maketitle

\begin{abstract}
    Consider a matroid $M$ whose ground set is equipped with a labeling to an abelian group.
    A basis of $M$ is called $F$-avoiding if the sum of the labels of its elements is not in a forbidden label set $F$.
    Hörsch, Imolay, Mizutani, Oki, and Schwarcz~(2024)
    conjectured that if an $F$-avoiding basis exists, then any basis can be transformed into an $F$-avoiding basis by exchanging at most $|F|$ elements.
    This \emph{proximity conjecture} is known to hold for certain specific groups; in the case where $|F| \le 2$; or when the matroid is subsequence-interchangeably base orderable (SIBO), which is a weakening of the so-called strongly base orderable (SBO) property.
    
    In this paper, we settle the proximity conjecture for sparse paving matroids or in the case where $|F| \le 4$.
    Related to the latter result, we present the first known example of a non-SIBO matroid.
    We further address the setting of multiple group-label constraints, showing proximity results for the cases of two labelings, SIBO matroids, matroids representable over a fixed, finite field, and sparse paving matroids.

\medskip


\end{abstract}

\section{Introduction}

Let $E$ be a finite ground set and let $\psi\colon E \to \Gamma$ be a labeling from $E$ to an abelian group $\Gamma$.
A \emph{group-label constraint} requires for a solution $X \subseteq E$ to satisfy $\psi(X) \coloneqq \sum_{e \in X} \psi(e) \notin F$, where $F\subseteq \Gamma$ is a prescribed set of forbidden labels.
Such a solution $X$ is called \emph{$F$-avoiding}.
An $F$-avoiding $X$ is also called \emph{zero} in the case when $F = \Gamma \setminus \set{0}$ (i.e., $\psi(X) = 0$), and \emph{non-zero} in the case when $F = \set{0}$ (i.e., $\psi(X) \ne 0$).
Several constraints in combinatorial optimization, such as parity, congruency, and exact-weight constraints, are representable as group-label constraints by letting $\Gamma$ be a cyclic group $\Z_m$ or the integers $\Z$, and $F$ be the complement of a singleton.
These constraints have been studied for many classical combinatorial optimization problems, including matching~\cite{papadimitriou1982exact,mulmuley1987matching,Artmann2017,elmaalouly2023exact,jia2023exact}, arborescence\cite{barahona1987exact}, submodular function minimization~\cite{goemans1995minimizing,nagele2019submodular}, minimum cut~\cite{nagele2020new}, and independent sets or bases in a matroid~\cite{camerini1992exact,doronarad2024lower,rieder1991lattices}.
Also, the non-zero and $F$-avoiding constraints have been particularly well-studied for path and cycle problems on graphs~\cite{chudnovsky2006apath,chudnovsky2008algorithm,iwata2022finding,kawase2020twoforbidpath,kobayashi2017finding,reed1999mangoes,wollan2010packing,wollan2011packing,thomas2023packing}.

In this paper, we study group-label constraints on matroid bases.
This line of research was initiated by Liu and Xu~\cite{liu2023congruency}, who addressed the problem of finding a zero basis. 
Hörsch, Imolay, Mizutani, Oki, and Schwarcz~\cite{horsch2024problems} considered non-zero bases, and more generally, $F$-avoiding bases, posing the following conjecture.

\begin{conjecture}[{Proximity Conjecture~\cite{horsch2024problems}}]\label{conj:proximity}
  Let $M$ be a matroid, $\psi \colon E \to \Gamma$ a labeling from the ground set $E$ of $M$ to an abelian group $\Gamma$, and $F \subseteq \Gamma$ a finite collection of forbidden labels.
  Then, for any basis $A$ of $M$, there exists an $F$-avoiding basis $B$ of $M$ such that $|A \setminus B| \le |F|$, provided that at least one $F$-avoiding basis exists.
\end{conjecture}

It is clear that \cref{conj:proximity} implies an algorithm for finding an $F$-avoiding basis using $O((rn)^{|F|})$ independence oracle queries, where $r$ is the rank of $M$ and $n \coloneqq |E|$.
Another consequence of \cref{conj:proximity} is that the number of $F$-avoiding bases is at least $|\mathcal{B}|/\sum_{i=0}^{|F|} \binom{r}{i}\binom{n-r}{i}$, where $\mathcal{B}$ is the basis family of $M$, provided that at least one $F$-avoiding basis exists.

\Cref{conj:proximity} is known to hold if (i) $|F| \le 2$,  (ii) $\Gamma$ is an ordered group, (iii) $\Gamma$ has prime order, (iv) $|\Gamma \setminus F| = 1$ and $\Gamma$ is a cyclic group with the order being a prime power or the product of two primes, or (v) $M$ is strongly base orderable (SBO). The claim for $|F| = 1$ essentially follows from Rieder's characterization~\cite{rieder1991lattices} of basis lattices, and a simpler proof can be found in~\cite{horsch2024problems}.
The proof for $|F| = 2$ in~\cite{horsch2024problems} first reduces the problem to 6-element, rank-3 matroids and then shows the claim for them, treating one special matroid $M(K_4)$ separately.
Case (ii) was also proven in~\cite{horsch2024problems}, where an ordered group is a group equipped with a total order consistent with the group operation, such as the integers $\Z$ and the reals $\R$. 

\Cref{conj:proximity} for (iii) and (iv) was shown by Liu and Xu~\cite{liu2023congruency}.
They showed (iii) with $|\Gamma \setminus F| = 1$ using an additive combinatorics result by Schrijver and Seymour~\cite{schrijver1990spanning} (which is \cref{conj:schrijver-seymour} below for prime-order cyclic groups), and it immediately extends to general $F$.
For (iv), Liu and Xu observed more generally that \cref{conj:proximity} for finite $\Gamma$ with $|\Gamma \setminus F| = 1$ holds if the following long-standing conjecture by Schrijver and Seymour~\cite{schrijver1990spanning}\footnote{%
  The original conjecture by Schrijver and Seymour~\cite{schrijver1990spanning} is stated only for the case when the stabilizer subgroup $H$ is trivial.
  \Cref{conj:schrijver-seymour} is the form given in~\cite{devos2009generalization} and is obtained by applying the original one to a labeling $\psi'\colon E \to \Gamma / H$ defined by $\psi'(e) \coloneqq \psi(e) + H$ for each $e \in E$.
}
is met for every subgroup of $\Gamma$.

\begin{conjecture}[{Schrijver and Seymour~\cite{schrijver1990spanning}}; see also~\cite{devos2009generalization}]\label{conj:schrijver-seymour}
    Let $M$ be a matroid with ground set $E$, basis family $\mathcal{B}$, and rank function $\rho$.
    Let $\psi\colon E \to \Gamma$ be a labeling to an abelian group $\Gamma$
    and $H \coloneqq \Set{g \in \Gamma}{g + \psi(\mathcal{B}) = \psi(\mathcal{B})}$ the stabilizer subgroup of $\psi(\mathcal{B}) \coloneqq \Set{\psi(B)}{B \in \mathcal{B}}$.
    Then,
    \begin{align}
        |\psi(\mathcal{B})| \ge |H| \prn*{\sum_{Q \in \Gamma / H} \rho(\psi^{-1}(Q)) - \rho(E) + 1}.
    \end{align}
\end{conjecture}

Schrijver and Seymour~\cite{schrijver1990spanning} showed \cref{conj:schrijver-seymour} for prime-order cyclic groups, and
DeVos, Goddyn, and Mohar~\cite{devos2009generalization} proved it for the cases when $M$ is obtained from a uniform matroid by adding parallel elements and when $\Gamma$ is one of the groups in (iv).

Case (v) was shown in~\cite{liu2023congruency} for $|\Gamma \setminus F| = 1$ and in~\cite{horsch2024problems} for general $F$.
SBO matroids are a class of matroids that admit a certain basis exchange property and includes gammoids (so in particular,  uniform, partition, laminar, and transversal matroids); see, e.g.,~\cite[Section~42.6c]{schrijver2003combinatorial}.
As mentioned in the full version~\cite{horsch2024problems_arxiv} of~\cite{horsch2024problems}, the same proof works if we only assume a weaker property called \emph{subsequence-interchangeable base orderability}
(cf.~\cref{lem:pigeon}).
Following Baumgart~\cite{baumgart2009ranking},
we say that a rank-$r$ matroid is \emph{subsequence-interchangeably base orderable} (SIBO) if every pair of bases $A$ and $B$ admits orderings $a_1, \dotsc, a_r$ of $A$ and $b_1, \dotsc, b_r$ of $B$ such that $(B \setminus \{b_i, \dotsc, b_j\}) \cup \{a_i, \dotsc, a_j\} = \set{b_1, \dotsc, b_{i-1}, a_i, \dotsc, a_j, b_{j+1}, \dotsc, b_r}$ is a basis for any pair $(i, j)$ with $1 \le i \le j \le r$. 
For each pair of bases, we call such a pair of orderings an \emph{SI-ordering}.
Baumgart~\cite{baumgart2009ranking} posed the following conjecture.

\begin{conjecture} [{Baumgart~\cite{baumgart2009ranking}}]\label{conj:baumgart}
    Every graphic matroid is SIBO.
\end{conjecture}

By case (v) above, \cref{conj:baumgart} would imply \cref{conj:proximity} for graphic matroids. Let us note that \cref{conj:baumgart} is a strengthening of the graphic matroid case of the following celebrated conjecture.

\begin{conjecture}[{Gabow~\cite{gabow1976decomposing}, see also \cite{wiedemann1984cyclic, cordovil1993bases}}]\label{conj:gabow}
    Let $A$ and $B$ be bases of a rank-$r$ matroid $M$.
    Then, there are orderings $a_1, \dotsc, a_r$ of $A$ and $b_1, \dotsc, b_r$ of $B$ such that $\set{a_1, \dotsc, a_i, b_{i+1}, \dotsc, b_r}$ and $\set{b_1, \dotsc, b_i, a_{i+1}, \dotsc, a_r}$ are bases for any $i = 1, \dotsc, r$.
\end{conjecture}

In contrast to \cref{conj:baumgart}, \cref{conj:gabow} is known to hold for graphic~\cite{wiedemann1984cyclic, kajitani1988ordering, cordovil1993bases}, and more generally for regular matroids~\cite{berczi2023reconfiguration}.

\paragraph{Our contributions.}
Our first main result is a proof of \cref{conj:proximity} for sparse paving matroids.
A rank-$r$ matroid is \emph{paving} if every circuit is of size either $r$ or $r+1$, and is \emph{sparse paving} if it and its dual are both paving.
Sparse paving matroids are significant in matroid theory as they have been used in hardness proofs for several algorithmic problems~\cite{doronarad2024lower,horsch2024problems_arxiv,jensen1982complexity,lovasz1981matroid} as well as counterexamples of conjectures.
In fact, sparse paving matroids were used in \cite{horsch2024problems} to disprove a strengthening of \cref{conj:proximity} for $|\Gamma \setminus F| = 1$ posed in the initial preprint version~\cite{liu2023congruencyarxiv} of \cite{liu2023congruency}.
Given this context, our positive result for sparse paving matroids provides additional evidence for \cref{conj:proximity}.
Moreover, since it is believed that asymptotically almost all matroids are sparse paving~\cite{mayhew2011on}, our result would imply that \cref{conj:proximity} holds in an asymptotic sense.

We also show \cref{conj:proximity} for $|F| \le 4$ using a computer-aided proof.
First, we use an observation from \cite{horsch2024problems_arxiv} (see \cref{lem:reduction}) to reduce to the case of matroids on at most 10 elements having rank at most 5.
By checking all such matroids using a SAT solver, it turns out that all of them are SIBO, except for a single matroid called $R_{10}$. We complete the proof by showing \cref{conj:proximity} for $R_{10}$ separately. For completeness, we also provide an elementary proof that $R_{10}$ is not SIBO. This is noteworthy, as it serves as the first example of a non-SIBO matroid and shows that \cref{conj:baumgart} does not extend to regular matroids.

As the second main thread of the paper, 
we consider an analog of \cref{conj:proximity} for multiple group labelings.
In this setting, given $k$ labelings $\psi_1, \dotsc, \psi_k$, where each $\psi_i$ is a map from $E$ to an abelian group $\Gamma_i$, and $k$ forbidden labels $f_1 \in \Gamma_1, \dotsc, f_k \in \Gamma_k$, we are to find a basis $B$ such that $\psi_i(B) \ne f_i$ for all $i \in \set{1, \dotsc, k}$.
Questions with similar constraints have also been studied for paths and cycles in graphs~\cite{huynh2019unified,gollin2022unifiederdhosposatheoremcycles,gollin2024unified,chekan2024half}.
Note that the $\psi_1 = \dotsb = \psi_k$ case corresponds to the single $F$-avoiding constraint with $F = \set{f_1, \dotsc, f_k}$.
We pose the following conjecture.

\begin{conjecture}[{Multi-Labeled Proximity Conjecture}] \label{conj:multiple} 
There is a computable function $d\colon \N \to \N$ such that for each $k \in \N$, matroid $M$ with ground set $E$, group labelings $\psi_i \colon E \to \Gamma_i$, group elements  $f_i \in \Gamma_i$ for $i = 1, \dotsc, k$, and basis $A$ of $M$, there exists a basis $B$ of $M$ with $\psi_i(B) \ne f_i$ for $i = 1, \dotsc, k$ and $|A\setminus B| \le d(k)$, provided that at least one such basis exists.
\end{conjecture}

As a lower bound, we show using uniform matroids 
that $d(k)$, if exists, must be at least $2^k - 1$.
\Cref{conj:proximity} for $|F| = 1$ implies that this is tight if $k=1$.
For $k = 2$, we show that $2^2 - 1 = 3$ is tight.
We further show that $d(k) = \lfloor (e-1/2)k! \rfloor - 1$ suffices for SIBO matroids.
We combine an extension of this result with a result of \cite{horsch2024problems} (\cref{thm:weakly}) to show the existence of such a function $d(k)$ 
for matroids representable over a fixed, finite field.
Finally, we prove an analogous result for sparse paving matroids using a similar method, but relying on a new structural observation on sparse paving matroids instead of \cref{thm:weakly}.

\paragraph{Related work.}
Eisenbrand, Rohwedder, and W\k{e}grzycki \cite{eisenbrand2024sensitivity} showed a proximity result on basis pairs of integer-labeled matroids.
For a matroid labeled with $\mathbb{Z}_m$ for a positive integer $m\ge 2$, this result implies that if there exists a zero basis,
then for any basis $A$, there exists a zero basis $B$ such that $|A\setminus B| = O(m^5)$.
This bound is weaker than the bound $|A\setminus  B|\le m-1$ implied by \cref{conj:proximity}.
Their result also implies a bound $|A\setminus B| = |\Gamma|^{O(\log |\Gamma|)}$ for a matroid labeled with a finite abelian group $\Gamma$.
The paper~\cite{eisenbrand2024sensitivity} additionally provided an FPT algorithm for finding a zero basis of a matroid labeled with a finite abelian group when parameterized by group size.

Hörsch, Imolay, Mizutani, Oki, and Schwarcz~\cite{horsch2024problems} also posed a weighted variant of \cref{conj:proximity}, which was settled for SBO matroids as well as the case when $|F| = 1$~\cite{horsch2024problems}.
Extending the proofs of \cref{conj:proximity} in other cases to the weighted conjecture is left for future work.

\paragraph{Organization.}
The rest of this paper is organized as follows.
\Cref{sec:preliminaries} describes preliminaries.
\Cref{sec:sparse-paving,sec:four-forbidden} prove \cref{conj:proximity} for sparse paving matroids and the case when $|F| \le 4$, respectively.
\cref{sec:multiple-label} gives proximity results in the setting of multiple labelings.
Finally, in \cref{sec:conclusion}, we conclude the paper with several open questions.

\section{Preliminaries}\label{sec:preliminaries}

For a nonnegative integer $k$ and a set $S$, let $[k] \coloneqq \set{1, \dotsc, k}$ and $\binom{S}{k} \coloneqq \Set{X \subseteq S}{|X| = k}$.
For a set $S$, $x \notin S$, and $y \in S$, we abbreviate $S \cup \set{x}$ as $S + x$ and $S \setminus \set{y}$ as $S - y$.
All groups are implicitly assumed to be abelian.
We use the additive notation for the group operation.
Let $\Z_m$ be the cyclic group of order $m$.
For a prime power $q$, let $\GF(q)$ be the finite field of size $q$.

We refer the readers to \cite{oxley2011matroid} for basic concepts and terminology in matroid theory. 
A matroid $M$ consists of a finite ground set $E(M)$ and a nonempty set family $\mathcal{B}(M)$ such that for any $B, B' \in \mathcal{B}(M)$ and $e \in B \setminus B'$, there exists $f \in B' \setminus B$ such that $B - e + f \in \mathcal{B}(M)$.
Every element in $\mathcal{B}(M)$ is called a \emph{basis} (or a \emph{base}).
The \emph{rank} of $M$ is the size of any basis of $M$. 
The \emph{dual} $M^*$ of $M$ is a matroid on the same ground set defined by $\mathcal{B}(M^*) = \Set{E(M) \setminus B}{B \in \mathcal{B}(M)}$.
For $X \subseteq E(M)$, the \emph{restriction} of $M$ to $X$, denoted by $M|X$, is a matroid on $X$ with $\mathcal{B}(M|X) = \Set{B' \in \binom{X}{r'}}{B' \subseteq B \ (\exists B \in \mathcal{B}(M))}$, where $r' \coloneqq \max_{B \in \mathcal{B}(M)} |B \cap X|$.
Also, the \emph{contraction} of $M$ by $X$ is a matroid $M/X \coloneqq (M^* | (E(M) \setminus X))^*$.
A matroid $M'$ is a \emph{minor} of $M$ if $M' = (M|X) / Y$ for some $X, Y$ with $Y \subseteq X \subseteq E(M)$.
Two matroids $M_1$ and $M_2$ are \emph{isomorphic} if there exists a bijection $\sigma \colon E(M_1) \to E(M_2)$ such that $\mathcal{B}(M_2) = \Set{\Set{\sigma(e)}{e \in B}}{B \in \mathcal{B}(M_1)}$.

A matroid $M$ of rank $r$ is called \emph{uniform} if $\mathcal{B}(M)=\binom{E(M)}{r}$; it is denoted by $U_{r, n}$ up to isomorphism, where $n = |E(M)|$.
A matroid $M$ is called \emph{$\mathbb{F}$-representable} if for some matrix $A$ over a field $\mathbb{F}$, $E(M)$ corresponds to the set of columns of $A$ and $\mathcal{B}(M)$ consists of the subsets of columns of $A$ each of which forms a basis of the vector space spanned by the columns of $A$.
We will use the following characterizations of paving and sparse paving matroids.
A matroid $M$ of rank $r$ is paving if and only if there exists a collection $\mathcal{H} = \set{H_1, \dotsc, H_k}$ of subsets of $E(M)$ such that $|H_i| \ge r$ for each $i \in [k]$, $|H_i \cap H_j| \le r - 2$ if $i \ne j$, and $\mathcal{B}(M) = \Set{B \in \binom{E(M)}{r}}{B \not\subseteq H_i \ (\forall i \in [k])}$ (cf.~\cite[Theorem~5.3.5]{frank2011connections}).
Also, $M$ is sparse paving if and only if there is such a representation with $|H_i| = r$ for each $i \in [k]$; in this case, $\mathcal{B}(M) = \binom{E(M)}{r} \setminus \mathcal{H}$ (cf.~\cite[Lemma~2.1]{bonin2013basis}). Note that the class of sparse paving matroids is minor-closed, since it is closed under taking restrictions and duals.

In an indirect approach to \cref{conj:proximity}, the following lemma is useful. It is similar to \cite[Corollary 4.4]{liu2023congruency} and is obtained from \cite[Lemmas~5.35 and 5.36]{horsch2024problems_arxiv}. 

\begin{lemma}[see~\cite{horsch2024problems_arxiv}] \label{lem:reduction}
Let $M$ be a matroid, $\psi\colon E(M)\to \Gamma$ a group labeling, and $F\subseteq \Gamma$ a finite set of forbidden labels.
Assume that $(M, \psi, F)$ is a counterexample to \cref{conj:proximity}, i.e., $M$ has an $F$-avoiding basis and it has a basis $A$ with $|A\setminus B|\ge |F|+1$ for any $F$-avoiding basis $B$.
Then, there exists a minor $M'$ of $M$ having rank $|F|+1$, a labeling $\psi'\colon E(M)\to \Gamma$, a set of labels $F'\subseteq \Gamma$ with $|F'| = |F|$, and a basis $B'$ of $M$ such that $B'$ is the only $F'$-avoiding basis of $M'$ and $E(M')\setminus B'$ is a basis.
\end{lemma}

The following observation was essentially noted in \cite[Remark~5.16]{horsch2024problems_arxiv}; we include a proof for completeness.

\begin{lemma}[see~\cite{horsch2024problems_arxiv}]\label{lem:pigeon}
    Let $E$ be a finite set, $\psi\colon E\to \Gamma$ a group labeling, $F\subseteq \Gamma$ a finite collection of labels, and $A=\set{a_1,\dots, a_r}$ and $B=\set{b_1,\dots, b_r}$ disjoint subsets of $E$, where $r=|F|+1$.
    If $B$ is $F$-avoiding, then there exists a pair $(i,j)$ with $1 \le i \le j \le r$ and $(i,j) \ne (1,r)$ such that $\hat{B}_{i,j}\coloneqq (B \setminus \set{b_i, \dotsc, b_j}) \cup \set{a_i, \dotsc, a_j} = \set{b_1,\dotsc, b_{i-1}, a_{i},\dotsc, a_j, b_{j+1},\dotsc, b_r}$ is $F$-avoiding.
    Furthermore, if $\psi(\hat{B}_{1,j}) \in F$ for every $j \ge 1$, then there exists a pair with $i > 1$ such that $\psi(\hat{B}_{i,j}) = \psi(B)$.
\end{lemma}
\begin{proof}
If there exists $1 \le k \le r$ such that $\psi(\hat{B}_{1,k}) \notin F$, then $(i, j) = (1, k)$ is indeed a desired pair.
Otherwise, as $|F| = r-1$, there exists a pair $(k_1, k_2)$ by the pigeonhole principle such that $1 \le k_1 < k_2 \le r$ and $\psi(\hat{B}_{1,k_1}) = \psi(\hat{B}_{1,k_2})$.
We then have $\psi(\hat{B}_{k_1+1,k_2}) = \psi(B) \notin F$, which means that the pair $(i, j) = (k_1 + 1, k_2)$ is a desired one.
\end{proof}

\section{Proximity Theorem for Sparse Paving Matroids}\label{sec:sparse-paving}
In this section, we prove the following theorem.

\begin{theorem}\label{thm:proximity-sparse-paving}
    \cref{conj:proximity} is true when $M$ is a sparse paving matroid.
\end{theorem}

\begin{proof}
Suppose, to the contrary, that \cref{conj:proximity} does not hold for a sparse paving matroid $M$ on the ground set $E$, a labeling $\psi \colon E \to \Gamma$, and a forbidden label set $F \subseteq \Gamma$.
Then, as minors of sparse paving matroids are sparse paving,
we may assume by \cref{lem:reduction} that $M$ has rank $r = |F|+1$, it contains exactly one $F$-avoiding basis $B$, and $E\setminus B$ is also a basis (with $\psi(E \setminus B) \in F$).
Clearly, $|F| \ge 1$.

We first consider the case when $|\psi(E)| \ge r + 1$.
Suppose that $B$ is \emph{rainbow}, i.e., no two elements in $B$ have the same label.
Fix an element $e \in E \setminus B$ such that $B + e$ is still rainbow, and consider the sets $(E \setminus B) - e + f$ for $f \in B$.
Since $M$ is sparse paving, at least $r-1 = |F|$ of these $r$ sets are bases, and none of them is $F$-avoiding, as $B$ is the only $F$-avoiding basis.
Since $B + e$ is rainbow, these bases and $E \setminus B$, each of which is obtained by adding an element of $B + e$ to $(E \setminus B) - e$, have distinct labels, a contradiction.

Next, suppose that $B$ is not rainbow.
As there are $r + 1$ elements of different labels, we can take a set $X \subseteq E$ with $|X| = r$ and an element $e \in E \setminus X$ such that $X + e$ is rainbow.
We consider $r$ sets $X + e - f$ for $f \in X$ in addition to $X$ itself.
Since $M$ is sparse paving, at least $r = |F| + 1$ of these $r + 1$ sets are bases, and none of them is $B$ as they are all rainbow.
Thus, none of them is $F$-avoiding as $B$ is the only $F$-avoiding basis.
Since $X + e$ is rainbow, these bases have distinct labels, a contradiction.

From now on, let us assume $|\psi(E)| \le r$.
To complete the proof, we use the following lemmas.
For each $g \in \Gamma$, we call $\psi^{-1}(g) \coloneqq \psi^{-1}(\set{g}) = \Set{e \in E}{\psi(e) = g}$ the \emph{label class} of $g$.

\begin{lemma}\label{lem:sparse-paving-union-lemma}
    For any $X \subseteq E$ with $|X| = r$, $\psi(X) \notin F$, and $X \ne B$, one of the following holds:
    \begin{enumerate}[{label={\upshape (\arabic*)}}]
        \item $X$ is the union of label classes, or
        \item $|B \setminus X| = 1$, $B \symdif X$ is a label class, and $X \cap B$ is the union of label classes.\label{item:sparse-paving-union-lemma-2}
    \end{enumerate}
\end{lemma}

\begin{proof}
    Assume that $X$ is not the union of label classes, i.e., there exists a pair of $e \in X$ and $f \notin X$ with $\psi(e) = \psi(f)$.
    Now $X$ is not a basis as $B$ is the only $F$-avoiding basis.
    Since $M$ is sparse paving, $X - e + f$ is a basis and $\psi(X - e + f) = \psi(X) \notin F$, which implies that $X - e + f = B$.
    This holds for any pair $(e, f)$ and the pair $(e, f)$ is indeed unique as $\{e, f\} = X \symdif B$, completing
    the proof.
\end{proof}

\begin{lemma}\label{lem:sparse-paving-union-lemma-2}
    One of the following holds:
    \begin{enumerate}[{label={\upshape (\arabic*)}}]
        \item $B$ is the union of label classes, or
        \item there exists $g \in \Gamma$ such that $|\psi^{-1}(g) \cap B| = |\psi^{-1}(g) \cap (E \setminus B)| = 1$ and $B \setminus \psi^{-1}(g)$ is the union of label classes.\label{item:sparse-paving-union-lemma-2-2}
    \end{enumerate}
\end{lemma}

\begin{proof}
    Assume that $B$ is not the union of label classes, i.e., there exists a pair of $e \in B$ and $f \notin B$ with $\psi(e) = \psi(f)$.
    Let $X = B - e + f$.
    Then, $X$ is not a basis as $\psi(X) = \psi(B) \notin F$ and $B$ is the only $F$-avoiding basis.
    Thus, \cref{lem:sparse-paving-union-lemma}\ref{item:sparse-paving-union-lemma-2} implies the statement~\ref{item:sparse-paving-union-lemma-2-2} here.
\end{proof}

Let $A \coloneqq E \setminus B$.
As in \Cref{lem:pigeon}, under fixed orderings $a_1, \dotsc, a_r$ of $A$ and $b_1, \dotsc, b_r$ of $B$, for each pair $(i, j)$ with $1 \le i \le j \le r$ and $(i, j) \neq (1, r)$, we define $\hat{B}_{i,j}\coloneqq (B \setminus \set{b_i, \dotsc, b_j}) \cup \set{a_i, \dotsc, a_j} = \set{b_1,\dotsc, b_{i-1}, a_{i},\dotsc, a_j, b_{j+1},\dotsc, b_r}$.

\begin{lemma}
\label{lem:coloring-ordering}
    Let $C$ be a set and $c\colon E \to C$ be a coloring.
    If $|c(A)| + |c(B)| \le r+1$, then there are orderings $a_1, \dotsc, a_r$ of $A$ and $b_1, \dotsc, b_r$ of $B$ such that $\hat{B}_{i,j}$ is not the union of color classes for any pair $(i, j)$ with $1 \le i \le j \le r$ and $(i, j) \neq (1, r)$. 
\end{lemma}
 
\begin{proof}
    We construct desired orderings by fixing $a_i$ and $b_i$ for each $i = 1, \dotsc, r - 1$ in this order so that there exists $a \in A \setminus \{a_1, \dotsc, a_i\}$ with $c(a) = c(a_i)$ or $b \in B \setminus \{b_1, \dotsc, b_i\}$ with $c(b) = c(b_i)$.
    We first show that this is sufficient for our purpose, and then show that this is indeed possible.

    Suppose to the contrary that for orderings satisfying the above condition, there exists a pair $(i, j)$ such that $\hat{B}_{i,j}$ is the union of color classes.
    If $j < r$, then there exists $k > j$ such that $c(a_k) = c(a_j)$ or $c(b_k) = c(b_j)$.
    Since $\hat{B}_{i,j}$ is the union of color classes, $\{a_j, a_k\}$ or $\{b_j, b_k\}$ must be included in or disjoint from $\hat{B}_{i,j}$, but $\{a_j, a_k, b_j, b_k\} \cap \hat{B}_{i,j} = \{a_j, b_k\}$ by definition, a contradiction.
    Otherwise, $j = r$ and hence $1 < i \ (\le j)$.
    Then, similarly, there exists $k \ge i$ such that $c(a_k) = c(a_{i-1})$ or $c(b_k) = c(b_{i-1})$, but $\{a_{i-1}, a_k, b_{i-1}, b_k\} \cap \hat{B}_{i,r} = \{a_k, b_{i-1}\}$, a contradiction.
    Thus, the above construction is sufficient.

    Now we show that we can fix $a_i$ and $b_i$ for each $i = 1, \dotsc, r - 1$ so that the above condition is satisfied.
    The proof is done by induction on $r$.
    The base case when $r = 1$ is trivial.

    Suppose that $r \ge 2$.
    If $|c(A)| + |c(B)| = r + 1$, then by the pigeonhole principle (as $|A| + |B| = 2r$), there exists an element $a \in A$ such that $c(a) \neq c(a')$ for any $a' \in A \setminus \{a\}$ or $b \in B$ such that $c(b) \neq c(b')$ for any $b' \in B \setminus \{b\}$.
    By symmetry, we may assume that $c(b) \neq c(b')$ for any $b' \in B \setminus \{b\}$, and then by the pigeonhole principle (as $|A| = r$ and $|c(A)| = r + 1 - |c(B)| \le r - 1$), there exist two distinct elements $a, a' \in A$ with $c(a) = c(a')$.
    In this case, by setting $a_1 = a$ and $b_1 = b$, the condition for $i = 1$ is satisfied as $a \in A \setminus \{a_1\}$ with $c(a) = c(a_1)$, and that for $i \ge 2$ can be satisfied by applying the induction hypothesis to the remaining part $A' = A \setminus \{a_1\}$, $B' = B \setminus \{b_1\}$, and $c' \colon (A' \cup B') \to C$, where $c'$ is the restriction of $c$ and satisfies $|c'(A')| + |c'(B')| = |c(A)| + |c(B)| - 1 = r$.
    
    The remaining case is when $|c(A)| + |c(B)| \le r$.
    Then, since $|c(A)| \le r - |c(B)| \le r - 1$, we have two distinct elements $a, a' \in A$ with $c(a) = c(a')$, and by setting $a_1 = a$ and arbitrarily choosing $b_1 \in B$, the condition for $i = 1$ is satisfied as $a \in A \setminus \{a_1\}$ with $c(a) = c(a_1)$ and that for $i \ge 2$ can be satisfied by applying the induction hypothesis as well.
    Thus, we complete the proof.
\end{proof}

We turn back to the proof of \cref{thm:proximity-sparse-paving}.
We discuss the two cases in \cref{lem:sparse-paving-union-lemma-2} separately.

\smallskip
(1)~
Suppose that $B$ is the union of label classes.
Let us regard $\psi$ as a coloring and apply \cref{lem:coloring-ordering} to get orderings $a_1, \dotsc, a_r$ of $A$ and $b_1, \dotsc, b_r$ of $B$.
We then apply \cref{lem:pigeon} to these orderings to get a pair $(i, j)$ with $1 \le i \le j \le r$ and $(i, j) \neq (1, r)$.
Then, $\hat{B}_{i,j}$ is $F$-avoiding but is not the union of label classes, contradicting to \cref{lem:sparse-paving-union-lemma}.

\smallskip
(2)~
Otherwise, there exists $g \in \Gamma$ such that $\psi^{-1}(g) \cap B = \set{e}$ and $\psi^{-1}(g) \cap A = \set{f}$ for some $e \in B$ and $f \notin B$.
We define a coloring $c\colon E \to \Gamma \cup \set{\star}$ by $c(e') \coloneqq \psi(e')$ for $e' \in E \setminus \set{f}$ and $c(f) \coloneqq \star$.
Now we apply \cref{lem:coloring-ordering} to $c$ to get orderings $a_1, \dotsc, a_r$ of $A$ and $b_1, \dotsc, b_r$ of $B$, and then by \cref{lem:pigeon}, we obtain a pair $(i, j)$ with $1 \le i \le j \le r$ and $(i, j) \neq (1, r)$ such that $\hat{B}_{i,j}$ is $F$-avoiding but is not the union of color classes.
We apply \cref{lem:sparse-paving-union-lemma} to $X = \hat{B}_{i,j}$.
(1) is impossible because the color classes refine the label classes.
In case of (2), by the assumption of \cref{lem:sparse-paving-union-lemma-2}(2), $B \symdif X = \set{e, f}$, implying that $X$ is the union of color classes, a contradiction.
\end{proof}

\section{Proximity Theorem for at Most 4 Forbidden Labels}\label{sec:four-forbidden}

We first observe that \cref{conj:baumgart} does not hold for regular matroids: a pair of disjoint bases of the matroid $R_{10}$ does not have an SI-ordering. 
$R_{10}$ is the matroid appearing in Seymour's fundamental decomposition theorem of regular matroids~\cite{seymour1980decomposition}, and it can be defined as the even-cycle matroid of the complete graph $K_5$.
The ground set of this matroid is the edge set of $K_5$ and its bases are the sets of five edges forming a subgraph containing exactly one odd cycle and no even cycle.
It is not difficult to check the following statement; see also the proof of \cite[Proposition~5.5]{berczi2023reconfiguration}.

\begin{lemma} \label{lem:r10auto}
	Consider $R_{10}$ as the even-cycle matroid of the complete graph $K_5$ on vertex set $\set{v_1,\dots, v_5}$.
	Then, for any two disjoint bases $A$ and $B$ of $R_{10}$, there exists an automorphism of $R_{10}$ mapping $A$ and $B$ to the 5-cycles $\Set{v_iv_{i+1}}{i \in [5]}$ and  $\Set{v_iv_{i+2}}{i \in [5]}$ of $K_5$, respectively, where indices are meant in a cyclic order (e.g., $v_6 = v_1$).
\end{lemma}

We show the following using \cref{lem:r10auto}.

\begin{restatable}{theorem}{RtenNotSIBO}\label{thm:R10notSIBO}
    $R_{10}$ is not SIBO.
\end{restatable}

\begin{proof}
We consider $R_{10}$ as the even-cycle matroid of the complete graph $K_5$ on vertex set $\set{v_1,\dots, v_5}$. 
    Let $A$ and $B$ be disjoint bases of $R_{10}$. By \cref{lem:r10auto}, we may assume $A=\Set{v_iv_{i+1}}{i \in [5]}$ and  $B=\Set{v_iv_{i+2}}{i \in [5]}$.
    Suppose to the contrary that there exists an SI-ordering $a_1,\dotsc, a_5$ of $A$ and $b_1,\dotsc, b_5$ of $B$.
    
	We may assume by symmetry that $a_1 = v_1v_2$. Using that $(a_1,b_1)$ is a symmetric exchange between $A$ and $B$, it follows that $b_1 = v_3v_5$.
	By symmetry, we may assume that $a_2 \in \set{v_2v_3, v_3v_4}$. As $(a_2,b_2)$ is a symmetric exchange between $A-a_1+b_1$ and $B-b_1+a_1$, it follows that \[(a_2,b_2)\in \set{(v_2v_3, v_2v_4), (v_3v_4, v_1v_3)}.\] If $(a_2, b_2) = (v_2v_3, v_2v_4)$, then $A-a_2 + b_2 = \set{v_1v_2, v_2v_4, v_3v_4, v_4v_5, v_5v_1}$ contains the 4-cycle $\set{v_1v_2, v_2v_4, v_4v_5, v_5v_1}$, a contradiction. We conclude that $(a_2, b_2) = (v_3v_4, v_1v_3)$.
 
	Using that $a_3 \in A\setminus \set{a_1,a_2} = \set{v_2v_3, v_4v_5, v_5v_1}$ and 	
	$(a_3, b_3)$ is a symmetric exchange between $(A\setminus \set{a_1,a_2})\cup \set{b_1,b_2} = \set{v_1v_3, v_3v_5, v_2v_3, v_4v_5, v_5v_1}$ and $(B\setminus \set{b_1,b_2})\cup \set{a_1,a_2} = \set{v_1v_2, v_3v_4,  v_2v_4, v_4v_1, v_5v_2}$, it follows that \[(a_3, b_3) \in \set{(v_2v_3, v_1v_2), (v_4v_5, v_1v_4), (v_1v_5, v_2v_5)}.\]
	It is clear that $(a_3, b_3) \ne (v_2v_3, v_1v_2)$ as $v_1v_2 = a_1$.
	If $(a_3,b_3) = (v_4v_5, v_1v_4)$, then $A-a_3+b_3 = \set{v_1v_2, v_3v_4, v_1v_4, v_2v_3, v_5v_1}$, which contains the 4-cycle $\set{v_1v_2, v_2v_3, v_3v_4, v_4v_1}$, a contradiction.
	Otherwise, $(a_3, b_3) = (v_1v_5, v_2v_5)$, and then $A-a_3+b_3 = \set{v_1v_2, v_3v_4, v_2v_5, v_2v_3, v_4v_5}$, which contains the 4-cycle $\set{v_2v_3, v_3v_4, v_4v_5, v_5v_2}$, a contradiction.
	This finishes the proof.
\end{proof}

Using a computer program, we verified that this is the only example of a basis pair not having an SI-ordering up to rank 5.

\begin{proposition} \label{prop:sat}
Let $M$ be a matroid of rank at most 5, and $(A,B)$ a basis pair of $M$ not having an SI-ordering.
Then, $A$ and $B$ are disjoint and the restriction $M|(A\cup B)$ is isomorphic to $R_{10}$.
\end{proposition}

Giving a human-readable proof of \cref{prop:sat} seems difficult, as even \cref{conj:gabow} was verified only up to rank 4~\cite{kotlar2013serial}. (Note that $R_{10}$ is known to satisfy \cref{conj:gabow}~\cite{berczi2023reconfiguration}, thus \cref{prop:sat} implies that the conjecture holds up to rank 5.)
Up to rank 4, one can check the validity of \cref{prop:sat} by using one of the existing databases of small matroids~\cite{mayhew2008nine}. 
As the list (or number) of rank-5 matroids on 10 elements is unknown and expected to be very large~\cite{mayhew2008nine}, we used a different approach: we encoded a basis pair of a matroid of given rank not having an SI-ordering as a Boolean formula, with variables encoding which subsets are bases, and decided the satisfiability with a SAT solver; see \cref{sec:sat} for details. We note that a similar but much more sophisticated approach has been used to study Rota's basis conjecture~\cite{kirchweger2022sat}.

We are ready to prove the following theorem.

\begin{theorem} \label{thm:four-forbidden}
\cref{conj:proximity} is true when $|F| \le 4$.
\end{theorem}
\begin{proof} 
Suppose otherwise and let $(M, \psi, F)$ be a counterexample. %with $|E(M)|$ being minimal.
We may assume by \cref{lem:reduction} that $M$ has rank $r = |F|+1$, it contains exactly one $F$-avoiding basis $B$, and $A \coloneqq E(M)\setminus B$ is also a basis (with $\psi(A) \in F$).

If the basis pair $(A, B)$ has an SI-ordering, then we obtain by \cref{lem:pigeon} an $F$-avoiding basis other than $B$, a contradiction.
Otherwise, as $M$ has rank $|F|+1\le 5$, \cref{prop:sat} implies that $M$ is isomorphic to $R_{10}$, which we regard as the even-cycle matroid of the complete graph $K_5$ on the vertex set $\set{v_1,\dots, v_5}$.
By \cref{lem:r10auto}, we may assume that $A=\Set{v_iv_{i+2}}{i \in [5]}$ and  $B=\Set{v_iv_{i+1}}{i \in [5]}$.
Fix $k \in [5]$ and define orderings
\begin{align}
(a_1,\dots, a_5) &\coloneqq  (v_{k+2}v_{k+4}, v_kv_{k+2}, v_{k+4}v_{k+6}, v_{k+3}v_{k+5}, v_{k+1}v_{k+3}) \\
(b_1,\dots, b_5) & \coloneqq (v_kv_{k+1}, v_{k+2}v_{k+3}, v_{k+3}v_{k+4}, v_{k+1}v_{k+2}, v_{k+4}v_{k+5})
\end{align}
of $A$ and $B$, respectively (see \cref{fig:r10}). Then, it can be checked that for indices $1 \le i \le j \le 5$, the set $\hat{B}_{i,j} \coloneqq \set{b_1,\dotsc, b_{i-1}, a_{i},\dotsc, a_j, b_{j+1},\dots, b_5}$ is a basis if and only if $(i,j)\ne (3,3)$.
In particular, $\hat{B}_{1,j}$ is a basis for each $j \ge 1$, and hence $\psi(\hat{B}_{1,j}) \in F$, as $B$ is the only $F$-avoiding basis.
Thus, by \cref{lem:pigeon}, we have $\psi(\hat{B}_{i,j}) = \psi(B) \notin F$ for some $(i, j)$, which must be $(3, 3)$ as $B$ is the only $F$-avoiding basis, again.
That is, $\psi(v_{k+3}v_{k+4}) = \psi(b_3) = \psi(a_3) = \psi(v_{k+4}v_{k+6})$.
Since this holds for every $k \in [5]$, we get that
\[\psi(B) = \sum_{k=1}^5 \psi(v_{k+3}v_{k+4}) = \sum_{k=1}^5 \psi(v_{k+4}v_{k+6}) = \psi(A),\]
which contradicts that $\psi(B) \notin F$ and $\psi(A) \in F$.

\begin{figure}[h!]
\centering
\includegraphics[width=0.38\textwidth]{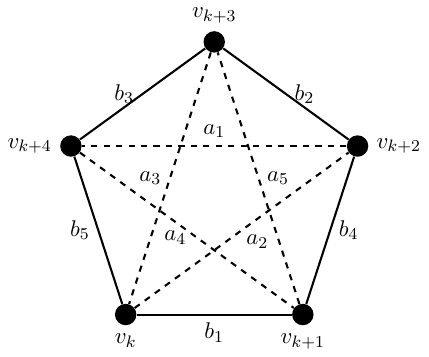}
\caption{Labeling the elements of the matroid $R_{10}$ such that for indices $0 \le i < j \le 5$, the set $\hat{B}_{i,j}$ is a basis if and only if $(i,j)\ne (2,3)$. Recall that bases are the sets of size 5 containing no $4$-cycle of $K_5$.} \label{fig:r10}
\end{figure}

\end{proof}

\section{Proximity for Multiple Labelings}\label{sec:multiple-label}

In this section, we verify \cref{conj:multiple} for various classes of matroids.
We begin with an example showing that the function $d$ in \cref{conj:multiple} must satisfy $d(k) \ge 2^k-1$ for each $k\ge 1$, even for uniform matroids.

\begin{example} \label{ex:tight}
    Let $k$ be a positive integer. Let $r = 2^k-1$, and let $U_{r,2r}$ denote the uniform matroid of rank $r$ on some ground set $E$ of size $2r$. We show that there exist group labelings $\psi_i \colon E \to \Gamma_i$ for $i \in [k]$, and a basis $A$ of $U_{r,2r}$ such that $B \coloneqq E \setminus A$ is the only basis of $M$ satisfying $\psi_i(B) \ne 0$ for all $i \in [k]$.
    
    Let us fix a set $A \subseteq E$ of size $r$.
    We set $\Gamma_{k} = \Z$, and $\Gamma_i = \Z_{2^i}$ for $i \in [k-1]$. Finally, we define the group labelings by
    \begin{align}
        \psi_i(e) \coloneqq
        \begin{cases}
            2^{i-1} - 1 \hspace{1em} &\text{if } e \in A, \\
            2^{i-1} \hspace{1em} &\text{if } e \notin A \text{ and } i \neq k, \\
            -2^{k-1} \hspace{1em} &\text{if } e \notin A \text{ and } i = k.  
        \end{cases}
    \end{align}
    
    Note that $r \equiv -1 \Mod{2^i}$ for all $i \in [k-1]$. 
    It is easy to see that $\psi_i(B) \neq 0$ for all $i \in [k]$, where $B = E \setminus A$.
    Hence, we only need to show that for each $\ell \in [r-1]$, any set $A'$ obtained by exchanging $\ell$ elements of $A$ into $B$ is zero in at least one of the labelings.
    
    Let $i \in \N$ be the largest for which $2^{i-1}$ divides $\ell +1$. Note that $\ell \equiv 2^{i-1} - 1 \Mod{2^i}$. Moreover, as $\ell < r = 2^k - 1$, we have $i \leq k$. Now if $i < k$, then
    \begin{align}
        \psi_i(A') &= (2^{i-1} - 1)(r - \ell) + 2^{i-1}\ell \\
        &= (2^{i-1}-1)r + \ell \\
        &\equiv (2^{i-1}-1) \cdot (-1) + (2^{i-1} - 1) &&\Mod{2^i} \\ & \equiv 0 &&\Mod{2^i}.
    \end{align} 
    On the other hand, if $i = k$, then $\ell = 2^{k-1} - 1$, and in this case \[\psi_k(A') = (2^{k-1}-1)(r - \ell) - 2^{k-1}\ell = (2^{k-1}-1)2^{k-1} - 2^{k-1}(2^{k-1}-1) = 0.\]
\end{example}

Next, we show the existence of $d(k)$ as in \cref{conj:multiple} for several matroid classes. The proofs will be based on the following lemma.

\begin{lemma} \label{lem:multiple}
Let $\psi_i\colon E \to \Gamma_i$ be group labelings on a finite set $E$ and $f_i \in \Gamma_i$ group elements for $i \in [k]$. Let $B \subseteq E$ be a subset with $\psi_t(B) \ne f_t$ for all $t \in [k]$.
Let $X_1, \dots, X_\ell \subseteq B$ and $Y_1, \dots, Y_\ell \subseteq E \setminus B$ be pairwise disjoint nonempty subsets. For $1 \le i \le j \le \ell$, let \[B_{i,j} \coloneqq (B\setminus (X_i \cups X_j)) \cup (Y_i \cups Y_j).\]
If $\ell \ge \lfloor (e-1/2) k!\rfloor$, there are indices $i, j$ such that $\psi_t(B_{i,j}) \ne f_t$ for all $t \in [k]$. 
\end{lemma}
\begin{proof}
Let $c_k$ denote the smallest integer such that whenever $\psi_1, \dots, \psi_k$ are labelings, $f_1, \dots, f_k$ group elements, and $B, X_1,\dots, X_\ell, Y_1,\dots, Y_\ell$ are subsets as in the theorem and $\ell \ge c_k$ holds, then there exist indices $1 \le i \le j \le \ell$ such that $\psi_t(B_{i,j}) \ne f_t$ for all $t \in [k]$. If no such integer exists, then define $c_k$ to be infinite.
Observe that $c_0 = 1$.

\begin{claim} \label{cl:recursion}
$c_k \le k \cdot c_{k-1}+1$ holds for $k \ge 1$.
\end{claim}
\begin{proof}
Assume for contradiction that $c_k > k \cdot c_{k-1} + 1$, that is, there exist labelings $\psi_1,\dots, \psi_k$, group elements $f_1,\dots, f_k$, and subsets $B, X_1,\dots, X_\ell, Y_1,\dots, Y_\ell$ as required with $\ell \ge k \cdot c_{k-1} + 1$ such that there exist no indices $1\le i \le j \le \ell$ such that $\psi_t(B_{i,j}) \ne f_t$ for all $t \in [k]$.
In particular, for each $q \in [\ell]$, there exists $t \in [k]$ such that $\psi_t(B_{1,q}) = f_t$. Since $\ell \ge k \cdot c_{k-1}+1$, this implies that there exists $t \in [k]$ with $|\Set{q \in [\ell]}{ \psi_t(B_{1,q}) = f_t}| \ge c_{k-1}+1$. We may assume that $t=k$, and let $1 \le q_0 < \dots < q_{c_{k-1}} \le \ell$ be indices with $\psi_k(B_{1,q_0}) = \dots = \psi_k(B_{1, q_{c_{k-1}}}) = f_k$. 
Define $X'_i \coloneqq X_{q_{i-1}+1} \cups X_{q_i}$ and $Y'_i \coloneqq Y_{q_{i-1}+1} \cups Y_{q_i}$ for $i \in [c_{k-1}]$, and $B'_{i,j} \coloneqq (B\setminus (X'_i\cups X'_j)) \cup (Y'_i\cups Y'_j)$ for $1\le i \le j \le c_{k-1}$. 
Observe that $Y'_i = B_{1, q_i} \setminus B_{1, q_{i-1}}$ and $X'_i = B_{1, q_{i-1}} \setminus B_{1, q_i}$ hold for $i \in [c_{k-1}]$; thus
\[
\psi_k(Y'_i) - \psi_k(X'_i) = \psi_k(B_{1, q_i}) - \psi_k(B_{1, q_{i-1}})  = f_k-f_k = 0.
\]
Using the definition of $c_{k-1}$, we get that there exist indices $1 \le i' \le j' \le c_{k-1}$ such that $\psi_t(B'_{i',j'}) \ne f_t$ for all $t \in [k-1]$.
By defining $i \coloneqq q_{i'-1}+1$ and $j \coloneqq q_{j'}$, we get $B_{i,j} = B'_{i',j'}$ and 
\[\psi_k(B_{i,j})  = \psi_k(B) + \sum_{s = i'}^{j'} (\psi_k(Y'_s)-\psi_k(X'_s)) = \psi_k(B) \ne f_k. 
\]
This proves that $\psi_t(B_{i,j}) \ne f_t$ for all $t \in [k]$, a contradiction.
\end{proof}

\cref{cl:recursion} implies that $c_1 \le c_0+1 = 2$ and $c_2\le 2c_1+1 \le 5$. We improve the latter bound by one.

\begin{claim}
\label{cl:2labelings}
$c_2 \le 4$.
\end{claim}

\begin{proof}
Assume for contradiction that $c_2 > 4$, and let $\psi_1,\dots, \psi_4$ be labelings, $f_1,\dots, f_4$ group elements, and $B, X_1,\dots,X_4, Y_1,\dots, Y_4$ subsets as in the proof of \cref{cl:recursion} such that there exist no indices $1\le i \le j \le 4$ such that $\psi_t(B_{i,j})\ne f_t$ for all $t \in [2]$.
We have that $|\Set{q \in [4]}{\psi_t(B_{1,q}) = f_t}| \le 2$ holds for all $t \in [2]$, as otherwise we get a contradiction using the proof of \cref{cl:recursion} and $c_1 \le 2$.
This implies that 
\begin{equation} \label{eq:prefix2}
|\Set{q \in [4]}{\psi_1(B_{1,q}) = f_1}| = |\Set{q \in [4]}{\psi_2(B_{1,q}) = f_2}| = 2. 
\end{equation}
By symmetry, we also get 
\begin{equation} \label{eq:suffix2}
 |\Set{q \in [4]}{\psi_1(B_{q,4}) = f_1}| = |\Set{q \in [4]}{\psi_2(B_{q,4}) = f_2}| = 2.
 \end{equation}
We may assume that $\psi_1(B_{1,4}) = f_1$. Let $p, q\in [3]$ denote the unique indices such that $\psi_1(B_{1,p}) = f_1$ and $\psi_1(B_{q+1,4}) = f_1$. If $p = q$, then \[f_1 + f_1 = \psi_1(B_{1,p}) + \psi_1(B_{p+1, 4}) = \psi_1(B) + \psi_1(B_{1,4}) = \psi_1(B)+f_1,\]
thus $\psi_1(B) = f_1$, a contradiction. We conclude that $p \ne q$. Let $r\in [3]$ be the unique index with $\set{p,q,r} = [3]$.
Then, from \eqref{eq:prefix2} and \eqref{eq:suffix2}, we get that $\psi_1(B_{1,p}) = \psi_1(B_{1,4}) = \psi_1(B_{q+1,4}) = f_1$ and $\psi_2(B_{1,q}) = \psi_2(B_{1,r}) = \psi_2(B_{p+1, 4}) = \psi_2(B_{r+1,4}) = f_2$.
Then,
\[f_2+f_2 = \psi_2(B_{1,r})+\psi_2(B_{r+1, 4}) = \psi_2(B_{1,q})+\psi_2(B_{q+1, 4}) = f_2 + \psi_2(B_{q+1,4}).\]
thus $\psi_2(B_{q+1, 4}) = f_2$. This contradicts \eqref{eq:suffix2} and finishes the proof.
\end{proof}

Using induction we obtain that $c_k \le  \left(\sum_{i=0}^k \frac{1}{i!} - \frac12\right) k!$.
Indeed, it holds for $k=2$ by \cref{cl:2labelings}, and using \cref{cl:recursion} and induction we get
\[c_k \le k \cdot c_{k-1}+1 \le k \cdot \left(\sum_{i=0}^{k-1} \frac{1}{i!} - \frac12\right) (k-1)!  + 1 = \left(\sum_{i=0}^k \frac{1}{i!}-\frac12\right) k!.\]
The statement of the lemma follows by using that $\sum_{i=0}^\infty \frac{1}{i!} = e$.
\end{proof}

\begin{theorem} \label{thm:mult}
Let $A$ be a basis of a matroid $M$ on ground set $E$, and $\psi_i\colon E \to \Gamma_i$ group labelings and $f_i \in \Gamma_i$ group elements for $i \in [k]$.
Assume that $M$ has at least one basis $B$ with $\psi_t(B) \ne f_i$ for all $i \in [k]$.
\begin{enumerate}[label=(\roman*)]
\item If $M$ is SIBO, then it has a basis $B$ with $\psi_t(B) \ne f_t$ for all $t \in [k]$ and $|A\setminus B| \le \lfloor (e-1/2)k!\rfloor-1$. \label{it:multSIBO}
\item If $k=2$, then $M$ has a basis $B$ with $\psi_1(A) \ne f_1$, $\psi_2(A) \ne f_2$, and $|A\setminus B| \le 3$. \label{it:mult2}
\end{enumerate}
\end{theorem}
\begin{proof}
Let $B$ be a basis of $M$ such that $\psi_i(B)\ne f_i$ for all $i \in [k]$ and $|A\setminus B|$ is minimum.
Assume that $\ell \coloneqq |A\setminus B| \ge \lfloor (e-1/2)k!\rfloor$. If $M$ is SIBO, then there exist orderings $x_1,\dots, x_\ell$ of $B$ and $y_1,\dots, y_\ell$ of $A$ such that $(B\setminus \set{x_i,\dots, x_j})\cup \set{y_i,\dots, y_j}$ is a basis for $1 \le i \le j \le r$. Then, using \cref{lem:multiple} with $X_i = \set{x_i}$ and $Y_i = \set{y_i}$ for $i \in [\ell]$, we get a contradiction to $A$ and $B$ being closest. This proves \ref{it:multSIBO}.
For proving \ref{it:mult2}, assume that $k=2$ and $M$ is not SIBO. Observe that in this case $|A\setminus B| \ge \lfloor (e-1/2) \cdot 2\rfloor  = 4$.
Let $A'$ be a basis of $M$ with $A'\subseteq A\cup B$ and $|A'\setminus B| = 4$.
Then, $M/(A'\cap B)$ has rank 4 and hence is SIBO by \cref{prop:sat}.
Thus, there exist orderings $x_1,\dots, x_4$ of $B\setminus A'$ and $y_1,\dots, y_4$ of $A' \setminus B$ such that $(B\setminus \set{x_i,\dots, x_j})\cup \set{y_i,\dots, y_j}$ is a basis for $1 \le i \le j \le 4$. This contradicts $A$ and $B$ being closest by \cref{lem:multiple}. 
\end{proof}

\cref{ex:tight} shows that for $k=2$, the bound 3 in \cref{thm:mult}\ref{it:mult2} is tight. Observe that this differs from the case $\psi_1 = \psi_2$ where the tight bound is 2~\cite{horsch2024problems}.

Finally, we derive the validity of \cref{conj:multiple} for matroids representable over a fixed, finite field and sparse paving matroids.
For positive integers $\alpha$ and $k$, we define a matroid $M$ to be \emph{weakly $(\alpha,k)$-base orderable} if for every ordered basis pair $(A,B)$ of $M$ with $|A\setminus B|\ge \alpha$, there exist pairwise disjoint nonempty subsets $X_1,\dots, X_k \subseteq B \setminus A$ and $Y_1, \dots, Y_k \subseteq A \setminus B$ such that $\left(B \setminus \bigcup_{i\in Z} X_i\right) \cup \bigcup_{i \in Z} Y_i$ is a basis for each $Z \subseteq [k]$. The following was shown in \cite{horsch2024problems}.

\begin{theorem}[Hörsch, Imolay, Mizutani, Oki, Schwarcz~\cite{horsch2024problems}] \label{thm:weakly}
    There is a computable function $h\colon \N\times \N \to \N$ such that for every prime power $q$, every $\GF(q)$-representable matroid is weakly $(h(q,k),k)$-orderable for any $k \in \N$.
\end{theorem}

Following the terminology of \cite{horsch2024problems}, for a basis $B$ of a matroid $M$, we say that a minor $M'$ of $M$ is a \textit{$B$-minor} if $M' = (M|X)/Y$ for some $X, Y$ with $Y \subseteq B \subseteq X \subseteq E(M)$.
In this case, $\mathcal{B}(M') = \Set{B' \subseteq X \setminus Y}{B' \cup Y \in \mathcal{B}(M)}$. We show the following result, which is related to results of \cite{pendavingh2018number}.

\begin{theorem}
\label{thm:uniform_minors}
    Let $k\ge 0$ be an integer and $M$ a sparse paving matroid of rank $r$. If $\min\{r, |E(M)|-r\} \ge \binom{2k}{k}$, then for each basis $B$, $M$ has a $B$-minor isomorphic to $U_{k, 2k}$.
\end{theorem}

\begin{proof}
We prove by induction on $k$. The statements clearly hold for $k=0$, so assume that it holds for $k-1$. Let $E\coloneqq E(M)$, $n \coloneqq |E|$, and $\cH\coloneqq \binom{E}{r} \setminus \mathcal{B}(M)$.
Recall that since $M$ is sparse paving, $|H_1 \cap H_2| \le r-2$ holds for all $H_1, H_2\in \cH$ with $H_1 \ne H_2$.
Since $\min\set{r, n-r} \ge \binom{2k-2}{k-1}$, by induction there exist $X, Y$ with $Y \subseteq B \subseteq X \subseteq E$ such that $(M|X) / Y$ is isomorphic to $U_{k-1, 2k-2}$.

Let $S\coloneqq X \setminus Y$ and $\cH_0 \coloneqq \Set{H\in \cH}{H \subseteq B\cup S,\ |H\cap S| = k}$.
Observe that $|Y \setminus H| = |Y|-|H|+|H\cap S| = (r-k+1) - r + k = 1$ for each $H\in \cH_0$.
This together with $M$ being sparse paving implies that for each set $Z\subseteq S$, there exists at most one $H\in \cH_0$ with $H\cap S = Z$, implying $|\cH_0| \le \binom{2k-2}{k}$.
Since $|Y \setminus H| = 1$ for $H \in \cH_0$ and $|Y| = r-k+1 \ge \binom{2k}{k}-k+1 > \binom{2k-2}{k} \ge |\cH_0|$, there exists $y\in Y \cap \bigcap_{H \in \cH_0} H$.
Let $Y'\coloneqq Y-y$.
We claim that $(M|X)/ Y'$ is isomorphic to $U_{k, 2k-1}$, i.e., for each subset $Z\subseteq S + y$ with $|Z|=k$, $Z \cup Y'$ is a basis of $M$.
Indeed, if $y \in Z$, then $Z \cup Y'$ being a basis follows from $(M|X) / Y$ being isomorphic to $U_{k-1, 2k-2}$.
If $y \notin Z$, then $|Z\cap S| = k$, thus $Z\not \in \cH$ follows from $y \in \bigcap_{H \in \cH_0} H$.
This shows that $(M|X)/ Y'$ is indeed isomorphic to $U_{k, 2k-1}$.

Let $T\coloneqq X \setminus Y'$ and $\cH_1\coloneqq \Set{H \in \cH}{Y' \subseteq H,\ |H\cap T| = k-1}$.
Observe that $|H\setminus X| = 1$ for each $H\in \cH_1$.
Then, as $M$ is sparse paving, for each set $Z\subseteq T$, there exists at most one $H\in \cH_1$ with $H\cap T =Z$, implying $|\cH_1| \le \binom{2k-1}{k-1}$.
Since $|H \setminus X|=1$ for $H\in \cH_1$ and $|E \setminus X| \ge |E|-r -k+1 \ge \binom{2k}{k}-k+1 > \binom{2k-1}{k-1} \ge |\cH_1|$, there exists $x \in E \setminus X$ with $x \in (E \setminus X) \setminus \bigcup_{H \in \cH_1} H$.
Let $X'\coloneqq X + x$.
We claim that $(M|X') / Y'$ is isomorphic to $U_{k, 2k}$, i.e., for each subset $Z\subseteq T + x$ with $|Z|=k$, $Z \cup Y'$ is a basis of $M$.
Indeed, if $x \notin Z$, then $Z \cup Y'$ being a basis follows from $(M|X)/ Y'$ being isomorphic to $U_{k, 2k-1}$.
If $x \in Z$, then $|Z\cap T| = k-1$, thus $Z \cup Y' \notin \cH$ follows from $x \not \in \bigcup_{H\in \cH_1} H$.
Therefore, $(M|X') / Y'$ is indeed isomorphic to $U_{k, 2k}$, finishing the proof.  
\end{proof}

\begin{remark}
We note that Pendavingh and van der Pol~\cite{pendavingh2018number} showed that for a fixed $k$, asymptotically almost all matroids contain $U_{k,2k}$ as a minor. If $M$ is a sparse paving matroid, then a counting argument found in \cite[Lemma~4.7]{pendavingh2018number} combined with the observation $|\mathcal{B}(M)| \ge \frac{r}{r+1} \binom{|E|}{r}$ implies that if $r \ge \binom{2k}{k}$ and $|E(M)|-r \ge k$ hold, then $M$ contains $U_{k,2k}$ as a minor. It is not clear whether a similar argument can be used to give a simple proof of the existence of such a $B$-minor for any basis $B$ as in \cref{thm:uniform_minors}.
\end{remark}

\cref{thm:weakly,thm:uniform_minors} and \cref{lem:multiple} immediately verify \cref{conj:multiple} for matroids representable over a fixed, finite field and sparse paving matroids.

\begin{corollary}
\label{cor:gfq-multiple}
    Let $\mathcal{M}$ be the class of (1) $\GF(q)$-representable matroids for a fixed prime power $q$ or (2) sparse paving matroids.
    Then, there is a computable function $d \colon \N \to \N$ such that if $M \in \mathcal{M}$, $\psi\colon E(M)\to \Gamma_i$ are group labelings, $f_i \in \Gamma_i$ are group elements for $i \in [k]$, and $A$ is a basis of $M$, then $M$ has a basis $B$ with $\psi_i(B) \ne f_i$ for all $i \in [k]$ and $|A\setminus B| \le d(k)$, provided that $M$ has at least one basis $B'$ with $\psi_i(B') \ne f_i$ for all $i \in [k]$.
\end{corollary}

\begin{proof}
(1)~
Suppose that $M$ is $\GF(q)$-representable.
Let $h$ be the function provided by \cref{thm:weakly}, and define $d(k)\coloneqq h(q, \lfloor (e-1/2) k!\rfloor)-1$. Let $B$ be a basis with $\psi_i(B) \ne f_i$ for all $i \in [k]$ such that $|A\setminus B|$ is minimum. If $|A\setminus B|>d(k)$, then by the definition of $h$, for $\ell \coloneqq \lfloor (e-1/2) k!\rfloor$, there exist pairwise disjoint nonempty subsets $X_1,\dots, X_\ell \subseteq B \setminus A$ and $Y_1, \dots, Y_\ell \subseteq A \setminus B$ such that $\left(B \setminus \bigcup_{i\in Z} X_i\right) \cup \bigcup_{i \in Z} Y_i$ is a basis for each $Z \subseteq [\ell]$. Then, we get a contradiction by \cref{lem:multiple} to $A$ and $B$ being closest.

(2)~
\cref{thm:uniform_minors} implies that sparse paving matroids are $(\binom{2k}{k}, k)$-weakly base orderable for each $k \ge 1$.
Therefore, as with Case (1), we get the desired function $d(k)$.
\end{proof}

\section{Conclusion}\label{sec:conclusion}

In this paper, we have proven \cref{conj:proximity} for the case when the matroid is sparse paving or $|F| \le 4$, and settled \cref{conj:multiple} for $k = 2$ and some classes of matroids.
We conclude this paper by posing new conjectures.

\begin{conjecture}\label{conj:sparse-paving-is-sibo}
    Every sparse paving matroid is SIBO.
\end{conjecture}

We have checked the validity of the conjecture up to rank 6 using a SAT solver.
If true, \Cref{conj:sparse-paving-is-sibo} would give another proof of \cref{thm:proximity-sparse-paving}.
Unfortunately, the proof~\cite{bonin2013basis} of \cref{conj:gabow} for sparse paving matroids does not seem to generalize to this conjecture.

We also pose the following refinement of \cref{conj:multiple} in light of our lower bound on $d(k)$.
We state it in the form of a question rather than a conjecture as we do not expect it to hold for general matroids, whereas it is more likely to hold for uniform, SBO, and SIBO matroids.

\begin{question}\label{conj:quantative-multiple}
    Let $M$ be a matroid with the ground set $E$, $\psi_i \colon E \to \Gamma_i$ a group labeling, and $f_i \in \Gamma_i$ a group element for $i \in [k]$.
    Then, if at least one such basis exists, for any basis $A$ of $M$, is there a basis $B$ of $M$ with $\psi_i(B) \ne f_i$ for $i \in [k]$ and $|A\setminus B| \le 2^k - 1$? 
\end{question}

\paragraph{Acknowledgments}
The authors are grateful to the organizers of the 16th Emléktábla Workshops, where the collaboration of the authors started.
The authors thank Kristóf Bérczi, Siyue Liu, and Chao Xu for several useful discussions.

\snote{Please update / add funding:} 
Ryuhei Mizutani was supported by JSPS KAKENHI Grant Number JP23KJ0379 and JST SPRING Grant Number JPMJSP2108.
Taihei Oki was supported by JST ERATO Grant Number JPMJER1903, JST CREST Grant Number JPMJCR24Q2, JST FOREST Grant Number JPMJFR232L, JSPS KAKENHI Grant Numbers JP22K17853 and 24K21315, and Start-up Research Funds in ICReDD, Hokkaido University.
Yutaro Yamaguchi was supported by JSPS KAKENHI Grant Numbers 20K19743 and 20H00605, and by Start-up Program in Graduate School of Information Science and Technology, Osaka University.
This research has been implemented with the support provided by the Lend\"ulet Programme of the Hungarian Academy of Sciences -- grant number LP2021-1/2021.

\bibliographystyle{abbrv}
\bibliography{main}

\newpage
\appendix
\section*{Appendix}

\section{CNF formulation of finding a non-SIBO matroid}\label{sec:sat}

In this section, we describe how we can reduce the problem of finding a $2r$-elements, rank-$r$, non-SIBO matroid to SAT by describing a CNF (conjunctive normal form) formulation.

Let $E = [2r]$ be the ground set.
We prepare $\binom{2r}{r}$ Boolean variables $x_B$ indexed by $B \in \binom{E}{r}$.
We build a CNF such that $\Set[\big]{B \in \binom{E}{r}}{\text{$x_B$ is true}}$ forms the basis family of a matroid, $[r]$ and $E \setminus [r]$ are bases, and $([r], E \setminus [r])$ has no SI-ordering by collecting the following clauses.
\begin{description}
    \item[Basis exchange property:] for every $A, B \in \binom{E}{r}$ and $e \in A \setminus B$,
    \begin{align}
        \neg x_A \vee \neg x_B \vee 
        \bigvee_{f \in B \setminus A} x_{A - e + f}.
    \end{align}

    \item[Fixed basis:]
    \begin{align}
        x_{[r]}.
    \end{align}

    \item[Fixed basis:]
    \begin{align}
        x_{E \setminus [r]}.
    \end{align}
    
    \item[No SI-ordering:] for every permutation $a_1, \dotsc, a_r$ of $[r]$ and $b_1, \dotsc, b_r$ of $E \setminus [r]$,
    \begin{align}
        \bigvee_{0 \le i < j \le r} 
          \neg x_{\set{a_1, \dotsc, a_i, b_{i+1}, \dotsc, b_j, a_{j+1}, \dotsc, a_r}}.
    \end{align}
\end{description}
Note that if a non-disjoint basis pair $(A, B)$ of a matroid $M$ has no SI-ordering, then $(A \setminus B, B \setminus A)$ has no SI-ordering as well in $M/(A\cap B)$.
Thus, we can restrict our attention to the disjoint basis pair $([r], E \setminus [r])$ by verifying the unsatisfiability of the CNF from small $r$.

Our Python script to solve the above SAT instance is available at \url{https://github.com/taiheioki/sibo}.

\end{document}